\documentclass[10pt,draft,twoside]{amsart}
\usepackage{amssymb}
\usepackage{latexsym}
\usepackage{amsfonts}
\usepackage{amsmath}

\newtheorem{theorem}{Theorem}[section]

\newtheorem{lemma}[theorem]{Lemma}

\theoremstyle{definition}

\newtheorem{remark}[theorem]{Remark}
\newtheorem{definition}[theorem]{Definition}

\renewcommand{\wr}{\,{\sf wr}\,}

\newcommand{\F}{\mathbb F}

\renewcommand{\leq}{\leqslant}
\renewcommand{\geq}{\geqslant}

\newcommand{\nei}{G_{\Gamma_1(C)}}

\newcommand{\Aut}{\mathrm{Aut}}
\newcommand{\Gr}{\mathrm{Aut}(\Gamma)}

\DeclareMathOperator{\Triples}{\sf{Triples}}
\DeclareMathOperator{\Pre}{Pre}
\DeclareMathOperator{\AGL}{AGL}
\DeclareMathOperator{\Perm}{Perm}
\DeclareMathOperator{\wt}{wt}
\DeclareMathOperator{\Fm2}{\F_2^{m/2}}
\DeclareMathOperator{\GL}{GL}

\begin{document}

\title{Neighbour
  Transitivity on Codes in Hamming graphs}
\author{Neil I. Gillespie and Cheryl E. Praeger}
\address{[Gillespie and Praeger] Centre for Mathematics of Symmetry and Computation\\
School of Mathematics and Statistics\\
The University of Western Australia\\
35 Stirling Highway, Crawley\\
Western Australia 6009}

\email{neil.gillespie@graduate.uwa.edu.au, cheryl.praeger@uwa.edu.au}

\begin{abstract}
We consider a \emph{code} to be a subset of the vertex set of a
\emph{Hamming graph}.  In this setting a \emph{neighbour} of the code
is a vertex which differs in exactly one entry from some codeword.
This paper examines codes with the property that some group of
automorphisms acts transitively on the \emph{set of neighbours} of the
code.  We call these codes \emph{neighbour transitive}.  We obtain
sufficient conditions for a neighbour transitive group to fix the code
setwise.  Moreover, we construct an infinite family of neighbour
transitive codes, with \emph{minimum distance} $\delta=4$, where this
is not the case.  That is to say, knowledge of even the complete set
of code neighbours does not determine the code.      
\end{abstract}

\thanks{{\it Date:} draft typeset \today
The first author is supported by an
  Australian Postgraduate Award and the second by Australian Research
  Council Federation Fellowship FF0776186.}

\maketitle

\section{Introduction}

We consider codes to be subsets of ordered $m$-tuples from a fixed
alphabet $Q$ of size $q$ and so it is natural to consider codes as
subsets of vertices of \emph{Hamming graphs} (see Section
\ref{secnot}).  In this setting a codeword in which exactly one of the
entries has been changed is adjacent in the Hamming graph to the
codeword and, provided it is not a codeword itself, we call it a
\emph{neighbour} of that codeword in the code.  For a code $C$ in a
Hamming graph, $\Gamma=H(m,q)$, the \emph{set of neighbours of $C$} is
the subset $\Gamma_1(C)$ consisting of the vertices of $\Gamma$ which
are not in $C$ but are joined by an edge to at least one element of
$C$.  In this paper we examine codes with the following property.     

\begin{definition}\label{def} Let $C$ be a code in
  $H(m,q)$.  Then we say that $C$ is \emph{neighbour transitive} if there
  exists a subgroup $X$ of the automorphism group of $H(m,q)$ that
  fixes setwise and acts transitively on the set of neighbours of
  $C$.  If we want to specify the group, we say $C$ is
  $X$-neighbour transitive. 
\end{definition}

\noindent Some much stronger group theoretic conditions than those
introduced here, such as \emph{complete transitivity} have been
studied previously in \cite{sole}, then in \cite{nonexist1},
\cite{nonexist2} and in a more general context in \cite{guidici}.     

In Definition \ref{def} it is not assumed that $X$ acts transitively on
$C$.  Indeed it was the question of whether the neighbour transitive
group $X$ was forced to fix $C$ setwise that led to the study and the
results of this paper.  The answer depends on the \emph{minimum
  distance}, $\delta$, of the code $C$.  This is the minimum number of
positions in which two distinct codewords from $C$ differ;
equivalently $\delta$ is the minimum distance in $H(m,q)$ between
distinct codewords in $C$. 

\begin{theorem}\label{thm}  Let $C$ be a code in $\Gamma=H(m,q)$ with
  $\delta\geq 3$, and let $x$ be an automorphism of $\Gamma$ fixing
  setwise the set $\Gamma_1(C)$ of neighbours of $C$.  Then at least
  one of the following holds:   
\begin{enumerate}
\item x fixes $C$ setwise, or
\item $\delta=4$, $q=2$ and $m$ is even,
\item $\delta=3$, and $m(q-1)$ is even.
\end{enumerate}
Moreover, for each even $m$ at least $4$, there exists an
$X$-neighbour transitive code $C\subset H(m,2)$ with minimum distance
$\delta=4$ such that $X$ does not fix $C$ setwise. 
\end{theorem}

In Section \ref{secnot} we introduce our notation, and provide some
interesting facts about codes in Hamming graphs.  In Section
\ref{precodeword} we introduce \emph{pre-codewords}.  Pre-codewords
are useful in proving our main assertions, which is done in Section
\ref{main}.  In Section \ref{inf} we introduce a family of codes with
$\delta=4$ and $m$ even, that have the property that for each code in
this family, the setwise stabiliser of the neighbour set in the
automorphism group of the Hamming graph does not fix the code setwise.
Moreover, we prove that each code in this family is neighbour
transitive.     

\subsection{Relevance to Error Correction for Codes}

Error correcting codes are used to maintain the integrity of data across
noisy communication channels and in storage systems.  The standard
scenario involves a message, or data, being transmitted over a noisy
communication channel.  The noise that is added during transmission
may result in errors occurring and the message changing.  Error correcting codes
are used to encode the message before transmission by adding a certain
amount of redundancy so that the likelihood of recovering the original
message is increased.  For a broad examination of error correcting
codes see \cite{macw1}, \cite{peterson} and \cite{verapless}.   

An assumption frequently made in decoding procedures for error
correcting codes is that the probability of a transmission error is
independent both of the position in which the error occurs, and also of
the incorrect letter of the alphabet occurring in that position, see
\cite[p.4]{peterson} and \cite[p.5]{verapless}.  In other words, the
probability of each error occurring is equally likely.  The concept of
a code being neighbour transitive is a group theoretic analogue of
this assumption.        

Studying codes that are neighbour transitive has led to unexpected
new constructions of codes with large minimum distance, along with
their automorphism groups \cite[Chapter 5]{ngthesis}, as well as new
classifications of some families of \emph{completely regular codes},
see \cite{hadpap}.      

\section{Notation}\label{secnot}

In this section we introduce the Hamming graphs and automorphisms of
the Hamming graph, and some notation that allows us to work with codes and
their neighbours.

\subsection{Hamming Graphs}  The Hamming graph with parameters $m$,
$q$, has a vertex set which consists of $m$-tuples with entries from a
set $Q$ of size $q$.  An edge exists between two vertices if and only
if they differ in precisely one entry.  We denote the Hamming graph by
$\Gamma=H(m,q)$.  In the Hamming graph, the \emph{Hamming distance} between
two vertices is defined to be the number of entries in which the two vertices
differ.  We use $d(\alpha,\beta)$ to denote the distance between the
vertices $\alpha$ and $\beta$.   

Sometimes we want to make reference to a particular vertex in
$H(m,q)$.  In order to do this we let $0$ be a distinguished element
of the alphabet.  This allows us to consider the zero vertex,
$(0,\ldots,0)={\mathbf{0}}$.  Let $\beta$ be a vertex of
$H(m,q)$.  Then we define the \emph{weight of $\beta$} to be the
number of non-zero entries of $\beta$, which we denote by
$\wt(\beta)$.  This is equal to the distance between $\mathbf{0}$ and
$\beta$ in $H(m,q)$.  

The automorphism group of the Hamming graph $H(m,q)$ is the 
semi-direct product $N\rtimes L$ where $N\cong S_q^m$ and $L\cong S_m$, see
\cite[Theorem 9.2.1]{distreg}. Throughout this paper we denote this
group by $G$, and for a subset $S$ of vertices in $H(m,q)$ we let
$G_S$ denote the setwise stabiliser in $G$ of $S$.  Let
$g=(g_1,\ldots,  g_m)\in N$, $\sigma\in L$ and
$\alpha=(\alpha_1,\ldots,\alpha_m)\in H(m,q)$. Then $g$ and $\sigma$
act on $\alpha$ in the following way:    
\begin{align*}
\alpha^g&=(\alpha_1^{g_1},\ldots,\alpha_m^{g_m})\\
\alpha^\sigma&=(\alpha_{1^{\sigma^{-1}}},\ldots,\alpha_{m^{\sigma^{-1}}}).
\end{align*}

\noindent It is straight forward to show that $G$ acts transitively on
the vertex set of $H(m,q)$.   

Let $\beta$ be a vertex in the Hamming graph.  We define the 
\emph{neighbours of $\beta$} to be the
set $$\Gamma_1(\beta)=\{\gamma\in H(m,q)\,|\,d(\beta,\gamma)=1\}.$$
We have the following general result about vertices of $H(m,q)$. 

\begin{lemma} \label{size2} Let $\alpha$ and $\beta$ be distinct vertices
  in $H(m,q)$.  Suppose that $d(\alpha,\beta)=2$.  Then
  $|\Gamma_1(\alpha)\cap\Gamma_1(\beta)|=2$. 
\end{lemma}
\begin{proof}  Since $G$ acts transitively on $H(m,q)$, we
  can assume without loss of generality that $\alpha=\mathbf{0}$.
  Therefore $\Gamma_1(\alpha)$ consists of all $m(q-1)$ weight one
  vertices.  As $d(\alpha,\beta)=2$, it follows that $\beta$
  has weight $2$.  Thus $\Gamma_1(\beta)$ consists of
  $(m-2)(q-1)$ vertices of weight $3$, $2(q-2)$ vertices of weight $2$
  and exactly $2$ vertices of weight $1$.  Thus
  $|\Gamma_1(\alpha)\cap\Gamma_1(\beta)|=2$.  
\end{proof}

Let $\Triples=\{(\alpha,\nu,\beta)\,|\, d(\alpha,\beta)=2$ and
$\nu\in\Gamma_1(\alpha)\cap\Gamma_1(\beta)\}.$ Let $y\in G$ act on
$(\alpha,\nu,\beta)\in \Triples$ as follows:
$(\alpha,\nu,\beta)^y=(\alpha^y,\nu^y,\beta^y)$.  

\begin{lemma}\label{trip} $G$ acts transitively on the set $\Triples$.
\end{lemma}

\begin{proof} Let $0$ and $1$ be distinguished elements of $Q$, and let
  $(\alpha,\nu,\beta)$ be an element of $\Triples$.  Since $G$ acts
  transitively on the vertices of $H(m,q)$, an arbitrary triple can be
  mapped to one with first entry $\mathbf{0}$, so we can assume that
  $\alpha=\mathbf{0}$.  Therefore we can also assume that $\nu$ and
  $\beta$ 
  have weight 1 and 2 respectively.  
  Moreover, since $\nu\in\Gamma_1(\alpha)\cap\Gamma_1(\beta)$, $\beta$
  has a non-zero entry in common with $\nu$, that is common in entry
  $i$ position and in the element $a_i$ of $Q$.  Thus    
\begin{align*}
\alpha&=(0,\ldots,0)\\
\nu&=(0,\ldots,a_i,\ldots,0)\\
\beta&=(0,\ldots,a_i,\ldots,a_j,\ldots,0).
\end{align*}
The stabiliser of $\alpha$ in $G$ is isomorphic to the wreath product
$S_{q-1}\wr S_m$.  Consider a group element
$g\sigma=(g_1,\ldots,g_m)\sigma\in G_\alpha$,  
where $a_i^{g_i}=1, a_j^{g_j}=1$, and $i^\sigma=1, j^\sigma=2$.  Then
$$(\alpha,\nu,\beta)^{g\sigma}=((0,\ldots,0),(1,0,\ldots,0),(1,1,0,\ldots,0)).$$
As we can map any element of $\Triples$ to this particular element, it
follows that $G$ acts transitively on $\Triples$. 
\end{proof}

\subsection{Codes in $H(m,q)$}\label{codes}  

As mentioned before, we consider codes in $H(m,q)$ to be subsets of the
vertex set.  Let $C$ be a code in $H(m,q)$.  We define the
\emph{minimum distance of C} to
be  $$\delta=\min\{d(\alpha,\beta)\,|\,\alpha,\beta\in C, 
\alpha\neq\beta\}.$$  

\noindent
We define the \emph{set of neighbours of $C$} to be the set
$\Gamma_1(C)=(\cup_{\alpha\in C}\Gamma_1(\alpha))\backslash C$.  Observe
that if $\delta\geq 2$ then $\Gamma_1(C)=\cup_{\alpha\in
  C}\Gamma_1(\alpha)$, and if $\delta\geq 3$, this is a disjoint union.

Recall $G=\Aut(\Gamma)$.  We define the \emph{automorphism group of
  $C$} to be the setwise stabiliser of $C$ in $G$, which we denote by
$\Aut(C)$.  Traditionally coding theorists regard certain weight
preserving subgroups of $\Aut(C)$ as the automorphism group of $C$
(see \cite[Sec. 1.6-1.7]{huffpless} for a nice explanation).  However,
because we are interested in groups of automorphisms acting
transitively on neighbour sets of codes, which may contain vertices of
different weights, we use this more general notion of $\Aut(C)$.  The
following result shows that $\Gamma_1(C)$ is necessarily
$\Aut(C)$-invariant.           

\begin{lemma}\label{autc=autn} Let $C$ be a code in $H(m,q)$.  Then
  $\Aut(C)\leq\nei$.  
\end{lemma}

\begin{proof} Let $\nu\in \Gamma_1(C)$.  Then there exists $\alpha\in
  C$ such that $d(\nu,\alpha)=1$.  Let $x\in\Aut(C)$.  Because
  adjacency is preserved by automorphisms it follows that
  $d(\nu^x,\alpha^x)=1$, and so $d(\nu^x,C)\leq 1$.  Suppose $\nu^x\in
  C$.  Then because $x\in\Aut(C)$ it follows that
  $\nu=(\nu^x)^{x^{-1}}\in C$, which is a contradiction.  Thus
  $\nu^x\in\Gamma_1(C)$.  
\end{proof}

Suppose that $C$ is $X$-neighbour transitive (see Definition
\ref{def}).  Since $X$ and $\Aut(C)$ leave $\Gamma_1(C)$ invariant, we
have that $C$ is also $\langle X,\Aut(C)\rangle$-neighbour
transitive, where $\langle X,\Aut(C)\rangle$ is the group generated by
$X$ and $\Aut(C)$.  Therefore we may assume that $\Aut(C)\leq X$.  The
question that this paper addresses is: \emph{when does $\Aut(C)=X$?}    

We are interested in distinguishing between codes which are
inherently different.  Therefore we introduce the following concept. 

\begin{definition} Let $C$ and $C'$ be two codes in $H(m,q)$.  We say
  that $C$ and $C'$ are \emph{equivalent} if there exists an
  automorphism $y$ of $H(m,q)$ such that $C^y=C'$.
\end{definition}

\noindent Equivalence preserves several important properties.

\begin{lemma}\label{equiv} Let $C$ be a code in $H(m,q)$ with minimum
  distance $\delta$, and let $y\in G$.  Then $C^y$ has minimum
  distance $\delta$ and $\Aut(C^y)=y^{-1}\Aut(C)y$.  Moreover, if $C$
  is $X$-neighbour transitive then $C^y$ is $(y^{-1}Xy)$-neighbour transitive.  
\end{lemma}

\begin{proof}  As automorphisms preserve distance in $H(m,q)$, it
  follows that if $C$ has minimum distance $\delta$, so too does
  $C^y$.  It is straight forward to prove that
  $y^{-1}\Aut(C)y=\Aut(C^y)$.  Now suppose that $C$ is $X$-neighbour
  transitive.  It is clear that $y^{-1}Xy$ fixes the neighbours of
  $C^y$ setwise.  Let $\nu_1$ and $\nu_2$ be neighbours of $\alpha_1$
  and $\alpha_2$, respectively, in $C^y$.  Then $\nu_1^{y^{-1}}$ and
  $\nu_2^{y^{-1}}$ are neighbours of $\alpha_1^{y^{-1}}$ and
  $\alpha_2^{y^{-1}}$, respectively, in $C$. Since $C$ is
  $X$-neighbour transitive, there exists $x\in X$ such that
  $\nu_1^{y^{-1}x}=\nu_2^{y^{-1}}$, and so $\nu_1^{y^{-1}xy}=\nu_2$. 
\end{proof}

\section{Pre-codewords}\label{precodeword}

Let $C$ be a code in $H(m,q)$.  The main investigation of this paper
is to determine when the setwise stabiliser in $G$ of $\Gamma_1(C)$
fixes $C$ setwise.  In this section we introduce the concept of a 
\emph{pre-codeword} which enables us to examine this question
further.  Firstly we consider the case where the setwise stabiliser of
the neighbours does not fix the code setwise.

\begin{lemma}\label{exist2}  Let $C$ be a code with $\delta\geq 3$.
  Suppose there exists $\alpha\in C$ and $y\in\nei$ such
  that $\alpha^y\notin C$.  Then for all $\nu\in\Gamma_1(\alpha)$,
  there exists a unique vertex $\pi\in\Gamma_2(\alpha)$ such that $\pi^y\in
  C$ and $\nu\in\Gamma_1(\pi)$.  
\end{lemma}

\begin{proof}  Note that $\Gamma_1(C)=\cup_{\beta\in
    C}\Gamma_1(\beta)$ since $\delta\geq 3$.  Let
  $\nu\in\Gamma_1(\alpha)$.  Since $y\in \nei$ it follows
  that $\nu^y\in\Gamma_1(C)$.  Hence there exists $\beta\in C$ such
  that $\nu^y\in\Gamma_1(\beta)$.  Let $\pi=\beta^{y^{-1}}$.  Then
  $\nu\in\Gamma_1(\pi)$.  It follows that $d(\alpha,\pi)\leq 2$.
  Moreover $\pi\neq\alpha$ since $\pi^y=\beta\in C$, while
  $\alpha^y\notin C$.  Consequently $\pi\notin C$ as $\delta\geq 3$.

Now suppose that $d(\alpha,\pi)=1$.  Then
$1=d(\alpha^y,\pi^y)=d(\alpha^y,\beta)$ and since $\beta\in C$ and
$\alpha^y\notin C$ this implies that $\alpha^y\in\Gamma_1(C)$.  Since
$y$ fixes $\Gamma_1(C)$ setwise it follows that $\alpha\in\Gamma_1(C)$
which is a contradiction since $\alpha\in C$.  Thus
$\pi\in\Gamma_2(\alpha)$, and so $\pi$ has all the required properties.
Suppose there exists $\pi'\neq\pi$ with $\pi'\in\Gamma_2(\alpha)$ such
that $\nu\in\Gamma_1(\pi')$ and $\pi'^y\in C$. Then
$d(\pi^y,\pi'^y)\leq 2$ contradicting the fact that $\delta\geq 3$.
\end{proof}

Thus for every neighbour $\nu$ of $\alpha$ there exists a unique
vertex $\pi\notin C$ adjacent to $\nu$ which gets mapped into $C$ by
$y$.  However, these vertices depend on $\alpha$ and the particular
$y$.  We now introduce the concept of a \emph{pre-codeword}. 

\begin{definition}\label{predef} Let $C$ be a code with $\delta\geq
  3$.  Let $\alpha$ be a codeword and
  $y\in\nei$ such that $\alpha^y\notin C$.  Then a \emph{pre-codeword}
  of $\alpha$ with respect to $y$ is a vertex $\pi$ such that 
  $d(\alpha,\pi)=2$ and $\pi^y\in C$.  We denote the set of all
    pre-codewords of $\alpha$ with respect to $y$ by $\Pre(\alpha,y)$.
\end{definition}

By definition each pre-codeword of $\alpha$ with respect to $y$ is at
distance two from $\alpha$.  We now use properties of Hamming graphs
to determine the cardinality of $\Pre(\alpha,y)$. 

\begin{lemma}\label{pre}  Let $C$ be a code with $\delta\geq 3$ and
  let $\alpha\in C$ and $y\in \nei$ such that
  $\alpha^y\notin C$.  Then
\begin{itemize}
\item[{\rm{(i)}}] $\{\Gamma_1(\alpha)\cap\Gamma_1(\pi)\,|\,\pi\in
  \Pre(\alpha,y)\}$ forms a partition of $\Gamma_1(\alpha)$. 
\item[{\rm{(ii)}}] $|\Pre(\alpha,y)|=m(q-1)/2$, in particular $m(q-1)$
  is even; and 
\item[{\rm{(iii)}}] for each $\pi\in \Pre(\alpha,y)$,
  $\Gamma_1(\pi)\subset\Gamma_1(C)$.
\end{itemize}
\end{lemma}

\begin{proof} Let $\nu\in\Gamma_1(\alpha)$.  By Lemma \ref{exist2}
  there exists a unique vertex $\pi\in \Pre(\alpha,y)$ such that
  $\nu\in\Gamma_1(\pi)$.  Thus $\bigcup_{\pi\in
    \Pre(\alpha,y)}\Gamma_1(\alpha)\cap\Gamma_1(\pi)$ covers
  $\Gamma_1(\alpha)$, and each $\nu$ lies in a unique subset
  $\Gamma_1(\alpha)\cap\Gamma_1(\pi)$.  Hence
  $\{\Gamma_1(\alpha)\cap\Gamma_1(\pi)\,|\,\pi\in \Pre(\alpha,y)\}$ is
  a partition of $\Gamma_1(\alpha)$ and (i) holds.  Furthermore, by
  Lemma \ref{size2}, each cell of this partition comprises exactly two
  neighbours of $\alpha$.  Thus $2\times 
  |\Pre(\alpha,y)|=|\Gamma_1(\alpha)|=m(q-1)$, giving us (ii).  Let
  $\pi\in \Pre(\alpha,y)$ and suppose $\nu$ is a neighbour of $\pi$.
  Then $\nu^y$ is a neighbour of the codeword $\pi^y$, that is
  $\nu^y\in\Gamma_1(C)$.  As $y$ stabilises $\Gamma_1(C)$ setwise, we
  have that $\nu\in\Gamma_1(C)$ and (iii) follows.   
\end{proof}

Let $\pi\in \Pre(\alpha,y)$.  Consider the set $C(\pi)=\{\beta\in
C\,|\,(\beta,\nu,\pi)\in\Triples\textnormal{ for some
  $\nu\in\Gamma_1(\pi)$}\}$.  We now demonstrate some similar results
for $C(\pi)$ to those for $\Pre(\alpha,y)$.      

\begin{lemma}\label{pre2} Let $\delta$, $\alpha$ and $y$ be as in Definition
  \ref{predef} and let $\pi\in \Pre(\alpha,y)$, and $C(\pi)$ as
  above. Then  
\begin{itemize}
\item[{\rm{(i)}}] $C(\pi)=\Gamma_2(\pi)\cap C$,
\item[{\rm{(ii)}}] $\{\Gamma_1(\beta)\cap\Gamma_1(\pi)\,|\,\beta\in
  C(\pi)\}$ is a partition of $\Gamma_1(\pi)$,
\item[{\rm{(iii)}}] $|C(\pi)|=m(q-1)/2$, in particular $m(q-1)$ is even, and
\item[{\rm{(iv)}}] $\beta^y\notin C$ for all $\beta\in C(\pi)$.
\end{itemize}
\end{lemma}

\begin{proof} Let $\beta\in C(\pi)$, thus $\beta\in C$.  By the definition of
  $\Triples$, $\beta\in\Gamma_2(\pi)$.  Thus
  $C(\pi)\subseteq\Gamma_2(\pi)\cap C$.  Conversely if
  $\beta'\in\Gamma_2(\pi)\cap C$, then there exists
  $\nu\in\Gamma_1(\pi)\cap\Gamma_1(\beta')$ by Lemma \ref{size2} and
  so $(\beta',\nu,\pi)\in \Triples$.  Thus $\beta'\in C(\pi)$.
  Therefore (i) holds.  From the definition of $C(\pi)$, $\cup_{\beta\in
    C(\pi)}(\Gamma_1(\beta)\cap\Gamma_1(\pi))=\Gamma_1(\pi)$.  If
  $\nu\in\Gamma_1(\beta)\cap\Gamma_1(\beta')\cap\Gamma_1(\pi)$ for
  $\beta,\beta'\in C(\pi)$, then $d(\beta,\beta')\leq 2$ and since 
  $\delta\geq 3$ it follows that $\beta=\beta'$.  Hence
  $\{\Gamma_1(\beta)\cap\Gamma_1(\pi)\,|\,\beta\in C(\pi)\}$ is a
  partition of $\Gamma_1(\pi)$, and (ii) holds.  Furthermore, by Lemma
  \ref{size2}, each cell of the partition comprises of exactly two
  neighbours of $\pi$.  Thus $2\times|C(\pi)|=|\Gamma_1(\pi)|=m(q-1)$, and
  (iii) follows.  In particular $m(q-1)$ is even.  Let
  $\beta\in C(\pi)$.  As $d(\pi^y,\beta^y)=2$ and 
  $\delta\geq 3$, we can conclude that $\beta^y\notin C$, so (iv) holds.  
\end{proof}

\section{Main Results}\label{main}

In this section we use the results from Section \ref{precodeword} to
find sufficient conditions under which $\nei$ fixes $C$ setwise.
Firstly we consider the case where $C$ has a large minimum distance.

\begin{lemma}\label{1} Let $C$ be a code with $\delta\geq 5$.  Then $C$ is
  fixed setwise by $\nei$.
\end{lemma}

\begin{proof} Since $\delta\geq 5$ it follows that $m\geq 5$.  Suppose
  $C$ is not fixed setwise by $\nei$.  Then by Lemma
  \ref{pre} there exist $\alpha\in C$ and $y\in\nei$ such that
  $|\Pre(\alpha,y)|\geq 3$.  Let $\pi_1$ and $\pi_2$ be distinct
  elements of $\Pre(\alpha,y)$.  Therefore
  $\pi_1,\pi_2\in\Gamma_2(\alpha)$ and $\pi_1^y, \pi_2^y\in C$. It
  follows that $d(\pi_1^y,\pi_2^y)=d(\pi_1,\pi_2)\leq 4$,
  contradicting the assumption that $\delta\geq 5$.  
\end{proof}

\noindent We now consider the case where $C$ has minimum distance of 4, but
also the size of $Q$ is at least $3$.

\begin{lemma}\label{2} Let $C$ be a code with $\delta=4$ and $q\geq 3$.
  Then $C$ is fixed setwise by $\nei$.
\end{lemma}

\begin{proof}  Suppose $C$ is not fixed setwise by $\nei$.
  Then there exist $\alpha\in C$ and $y\in \nei$ such that
  $\alpha^y\notin C$.  Since $q\geq 3$, we let $0,1$ and $2$ be
  distinct elements of $Q$.  Let $\pi\in \Pre(\alpha,y)$ and
  $\nu\in\Gamma_1(\alpha)\cap\Gamma_1(\pi)$.  By definition we know
  that $d(\alpha,\pi)=2$.  We also know, by Lemma \ref{trip}, that $G$
  is transitive on triples $(\alpha',\nu',\pi')$ with
  $\alpha',\nu',\pi'$ vertices of $H(m,q)$ such that $d(\alpha',\pi')=2$
  and $\nu'\in\Gamma_1(\alpha')\cap\Gamma_1(\pi')$. So replacing $C$ by
  an equivalent code if necessary, we may assume that $\alpha={\mathbf
    0}$, $\nu=(2,0,\ldots,0)$ and $\pi=(2,1,0,\ldots,0)$, and by Lemma
  \ref{equiv}, we can still assume that the minimum distance is
  $\delta=4$.  By part (iii) of Lemma \ref{pre},
  $\Gamma_1(\pi)\subseteq\Gamma_1(C)$.  Thus
  $\nu_2=(1,1,0,\ldots,0)\in\Gamma_1(C)$, and so $\nu_2$ is the
  neighbour of a codeword $\beta$, say.  It follows that $\beta$ must
  have weight either $1$,$2$ or $3$ and hence that
  $d(\alpha,\beta)\leq 3<\delta$, which is a contradiction.  
\end{proof}

\begin{lemma}\label{3} Let $C$ be a code with $\delta\geq 3$ with $q$ even and
  $m$ odd.  Then $C$ is fixed setwise by $\nei$.
\end{lemma}

\begin{proof}  Suppose $C$ is not fixed setwise by $\nei$.
  Then by Lemma \ref{pre} there exist $\alpha\in C$ and $y\in\nei$
  such that $2\times |\Pre(\alpha,y)|=m(q-1)$.  Thus $2$ divides either
  $m$ or $q-1$.  If $m$ is odd and $q$ is even this is not possible.
\end{proof}

\noindent Lemmas \ref {1}--\ref{3} together yield a proof that at least
one of $(1), (2)$ or $(3)$ of Theorem \ref{thm} holds.

\section{Infinite Family of Binary Codes $C$}\label{inf}

In this section we define a family of binary codes in $H(m,2)$ where
$m$ is even and at least 4.  For a code $C$ in this family, we prove
that $C$ has minimum distance $\delta=4$, is $\nei$-neighbour
transitive, and that $\nei$ does not fix $C$ setwise.  That is, the
final statement of Theorem \ref{thm} holds for this family of codes.     

We can view the vertex set of  $H(m,2)$ as the vector space $\F_2^m$
of $m$-dimensional row vectors over $\F_2$. With this in mind, for
each $i\in M=\{1,\ldots,m\}$,  we let $e_i$ denote the vertex with $1$
only in the $i^{th}$ position.  Furthermore, because the base group
$N\cong S_2^m$ of $G=\Gr=N\rtimes L\cong S_2\wr S_m$ is regular on the
vertices of $H(m,2)$, we may identify $N$ with the group of
translations of $\F_2^m$, and $G$ with a subgroup of the affine
group $\AGL(m,2)$.  More precisely $N$ consists of the translations
$\phi_\alpha$, where $\beta^{\phi_\alpha}=\beta+\alpha$ for
$\alpha,\beta\in \F_2^m$, and if $\mathbf{0}$ is the zero vector, then
$G=N\rtimes G_{\mathbf{0}}$ where $G_{\mathbf{0}}$ is the group of
permutation matrices in $\GL(m,2)$.  For any subset $S$ in $\F_2^m$ we
let $\Perm(S)$ denote the group of permutation matrices that fix $S$
setwise.   

Let $m$ be even and at least $4$.  Then we can consider vectors in
$\F_2^m$ as $2$-tuples of vectors from $\Fm2$.  In particular, for any
vertex $\alpha=(\alpha_1,\ldots,\alpha_m)\in \F_2^m$,  
we can identify $\alpha$ with the $2$-tuple $(\beta,\gamma)$ where
$\beta=(\alpha_1,\ldots,\alpha_{\frac{m}{2}})$, 
$\gamma=(\alpha_{\frac{m}{2}+1},\ldots,\alpha_m)\in\Fm2$.  Given this,
we define the following subsets of $\F_2^m$:\begin{align*}U=&\{(\beta,\beta)\in \F_2^m\,:\,\beta\in \Fm2\},\\
C=&\{(\beta,\beta)\in \F_2^m\,:\,\wt(\beta)\textnormal{ is even in $\Fm2$}\}.  
\end{align*}  It follows from the definitions that $U$ and $C$ are
subspaces of $\F_2^m$, and thus are linear codes.  The minimum weight of
vectors in $U$ and $C$ is $2$ and $4$ respectively.  Therefore the
minimum distance of $U$ and $C$ is $\delta_U=2$ and $\delta_C=4$
respectively.  Also, it is straight forward to deduce   
that  $$\Gamma_1(U)=\{(\beta,\gamma)\in
\F_2^m\,:\,d(\beta,\gamma)=1\},$$ where
$d(\beta,\gamma)$ is the Hamming distance between $\beta$ and $\gamma$
in $H(m/2,2)$.  We now show that the neighbour sets of $U$ and $C$
coincide.    
\begin{lemma}\label{neieq} $\Gamma_1(U)=\Gamma_1(C)$.
\end{lemma}
\begin{proof}  Because $\delta_U=2$ we have that $\Gamma_1(U)=\cup_{(\beta,\beta)\in
    U}\Gamma_1((\beta,\beta))$, and so
  $\Gamma_1(C)\subseteq\Gamma_1(U)$.  Conversely, suppose 
  $(\beta,\gamma)\in\Gamma_1(U)$.  Then $d(\beta,\gamma)=1$ in
  $H(m/2,2)$, and so $\wt(\beta)$ and $\wt(\gamma)$ have opposite
  parity. From this we conclude that either $(\beta,\beta)\in C$ or
  $(\gamma,\gamma)\in C$.  In either case it follows that
  $(\beta,\gamma)\in\Gamma_1(C)$.     
\end{proof}

Let $J_1=\{1,\ldots,\frac{m}{2}\}$ and
$J_2=\{\frac{m}{2}+1,\ldots,m\}$, and consider the partition  
$\mathcal{J}=\{J_1,J_2\}$ of $M$.  Let $H$ be the stabiliser of
$\mathcal{J}$ in $S_m$.  Then $H\cong S_{m/2}\wr S_2$, and a
typical element of $H$ is of the form
$(\sigma_1,\sigma_2)\hat{\sigma}$ where $\sigma_1,\sigma_2\in 
S_{\frac{m}{2}}$ and $\hat{\sigma}\in S_2$.  Let   
$K=\{(\sigma_1,\sigma_2)\hat{\sigma}\in H\,:\,\sigma_1=\sigma_2\}$.
Then $K\cong S_{\frac{m}{2}}\times S_2$, a transitive subgroup of
$S_m$ and we can identify $K$ with a subgroup of permutation matrices
in $\GL(m,2)$.  As such, for $y=(\sigma,\sigma)\hat{\sigma}\in K$ we have that
$(\beta,\beta)^y=(\beta^\sigma,\beta^\sigma)$ for all
$(\beta,\beta)\in U$.  Therefore $K$ stabilises $U$.  Furthermore, 
because permutation matrices in $\GL(m/2,2)$ preserve weights of
vectors in $\Fm2$, it follows that $K$ stabilises $C$ also.         

\begin{lemma}\label{cneitrans} $C$ is $\Aut(C)$-neighbour transitive.     
\end{lemma}

\begin{proof}  By Lemma \ref{autc=autn}, $\Aut(C)$ fixes
  $\Gamma_1(C)$ setwise.  Let $\nu_1,\nu_2\in \Gamma_1(C)$.  Then
  there exist $\alpha_i\in C$ such that $\nu_i\in\Gamma_1(\alpha_i)$ for 
  $i=1,2$.  It follows that the translation $\phi_{\alpha_i}$ 
  maps $\alpha_i$ to $\mathbf{0}$ for $i=1,2$.  Because
  adjacency is preserved by automorphisms of $\Gamma$, we have that
  $\nu_i^{\phi_{\alpha_i}}\in\Gamma_1(\mathbf{0})$ for
  $i=1,2$. Therefore there exists $s,t\in M$ such that 
  $\nu_1^{\phi_{\alpha_1}}=e_s$ and $\nu_2^{\phi_{\alpha_2}}=e_t$.
  Since $K\leq\Aut(C)$ acts transitively on $M$, and because
  permutation matrices preserve weight, it follows that there exists
  $\sigma\in K$ such that $e_s^\sigma=e_t$.  Hence
  $\nu_1^{\phi_{\alpha_1}\sigma\phi_{\alpha_2}}=\nu_2$.  Finally,
  because $\alpha_i\in C$ and $C$ is linear, we have
  $\phi_i\in\Aut(C)$, and also $\sigma\in K\leq\Aut(C)$, so
  $\phi_1\sigma\phi_2\in\Aut(C)$.         
\end{proof}

Since $U$ is a binary linear code, we can conclude from \cite[Lemma
  3.1]{guidici} that $\Aut(U)=N_U\rtimes\Perm(U)$, where $N_U$ is the  
group of translations generated by $U$.  By Lemma \ref{autc=autn},
$\Aut(U)$ fixes $\Gamma_1(U)$ setwise, and by Lemma \ref{neieq},
$\Gamma_1(U)=\Gamma_1(C)$.  Therefore $\Aut(U)\leq \nei$.
Since $N_U$ does not fix $C$ setwise it follows that $\nei$
does not fix $C$ setwise.  Furthermore, Lemma \ref{autc=autn} and
Lemma \ref{cneitrans} imply that $\nei$ acts transitively on
$\Gamma_1(C)$.  Hence $C$ is $\nei$-neighbour transitive but is not
fixed setwise by $\nei$.  Thus we have proved the final statement of 
Theorem \ref{thm}.  

\begin{remark}  Equivalent codes to the ones described in
  this section can be found in \cite[Chapter 3]{ngthesis}, in which
  the automorphism groups of these codes are given.  From this we can
  conclude that for $m\geq 6$, $\Aut(U)=N_U\rtimes K'$ and
  $\Aut(C)=N_C\rtimes K'$ where $K'\cong S_2\wr S_{m/2}$. 
  Furthermore, $\nei=N_U\rtimes K'$.  The case where $m=4$
  is an exception.  In this case $C=\{\mathbf{0}, 
  \mathbf{1}\}$ (where $\mathbf{1}=(1,1,1,1)$) and $\Aut(C)=N_C\rtimes
  G_{\mathbf{0}}$.  However, $\Aut(U)=N_U\rtimes K'$ where $K'\cong D_8$
  and $\nei=N_W\rtimes G_{\mathbf{0}}$ where $W$ is the
  subspace of $\F_2^4$ consisting of all even weight vectors.  Again
  it follows that $C$ is $\nei$-neighbour transitive
  but $\nei$ does not fix $C$ setwise.  Interesting, however, because
  $\Gamma_1(C)=\Gamma_1(U)$, in this case we also have that $U$ is
  $G_{\Gamma_1(U)}$-neighbour transitive but $G_{\Gamma_1(U)}$ does
  not fix $U$ setwise.     
\end{remark}

\section[]{Acknowledgement}
The authors would like to thank Professor Alice Niemeyer for her
assistance in the preparation of this paper.  The authors would also
like to thank the referees for their comments which helped improve the
exposition.  


\begin{thebibliography}{10}
\providecommand{\url}[1]{{#1}}
\providecommand{\urlprefix}{URL }
\expandafter\ifx\csname urlstyle\endcsname\relax
  \providecommand{\doi}[1]{DOI~\discretionary{}{}{}#1}\else
  \providecommand{\doi}{DOI~\discretionary{}{}{}\begingroup
  \urlstyle{rm}\Url}\fi

\bibitem{nonexist1}
Borges, J., Rif{\`a}, J.: On the nonexistence of completely transitive codes.
\newblock IEEE Trans. Inform. Theory \textbf{46}(1), 279--280 (2000).
\newblock \doi{10.1109/18.817528}

\bibitem{nonexist2}
Borges, J., Rif{\`a}, J., Zinoviev, V.: Nonexistence of completely transitive
  codes with error-correcting capability {$e>3$}.
\newblock IEEE Trans. Inform. Theory \textbf{47}(4), 1619--1621 (2001)

\bibitem{distreg}
Brouwer, A.E., Cohen, A.M., Neumaier, A.: Distance-Regular Graphs,
  \emph{Ergebnisse der Mathematik und ihrer Grenzgebiete (3) [Results in
  Mathematics and Related Areas (3)]}, vol.~18.
\newblock Springer-Verlag, Berlin (1989)

\bibitem{ngthesis}
Gillespie, N.I.: Neighbour Transitivity on Codes in Hamming graphs.
\newblock Ph.D. thesis, The University of Western Australia, Perth, Australia
  (2011)

\bibitem{hadpap}
Gillespie, N.I., Praeger, C.E.: Uniqueness of certain completely regular
  Hadamard codes.
\newblock Preprint (2011)

\bibitem{guidici}
Giudici, M., Praeger, C.E.: Completely transitive codes in {H}amming graphs.
\newblock European J. Combin. \textbf{20}(7), 647--661 (1999)

\bibitem{huffpless}
Huffman, W.C., Pless, V.: Fundamentals of Error-Correcting Codes.
\newblock Cambridge University Press, Cambridge (2003)

\bibitem{macw1}
MacWilliams, F.J., Sloane, N.J.A.: The Theory of Error-Correcting Codes.
\newblock North-Holland Publishing Co., Amsterdam (1977)

\bibitem{peterson}
Peterson, W.W., Weldon Jr., E.J.: Error-Correcting Codes, second edn.
\newblock The M.I.T. Press, Cambridge, Mass.-London (1972)

\bibitem{verapless}
Pless, V.: Introduction to the Theory of Error-Correcting Codes, third edn.
\newblock Wiley-Interscience Series in Discrete Mathematics and Optimization.
  John Wiley \& Sons Inc., New York (1998)

\bibitem{sole}
Sol{\'e}, P.: Completely regular codes and completely transitive codes.
\newblock Discrete Math. \textbf{81}(2), 193--201 (1990)

\end{thebibliography}

\end{document}